\documentclass[reqno,12pt]{article}

\usepackage{amsmath}
\usepackage{amsfonts}
\usepackage{amssymb}
\usepackage{eufrak}
\usepackage{amscd}
\usepackage{amsthm}
\usepackage{epsfig}
\usepackage{amstext}
\usepackage[all]{xy}

\theoremstyle{plain}

\newtheorem{theorem}{Theorem}

\newtheorem{lemma} [theorem] {Lemma}
\newtheorem{remark}[theorem]{Remark}

\def\diam{\hskip0.02cm{\rm diam}\hskip0.01cm}

\newcommand{\WOT}{{\rm WOT}}

\usepackage{graphicx}
\usepackage{amsmath}
\usepackage{fullpage}

\newtheorem{corollary}[theorem]{Corollary}

\newtheorem{proposition}[theorem]{Proposition}

\numberwithin{theorem}{section} \numberwithin{equation}{section}

\newcommand{\E}{\mathcal{E}}

\newcommand{\X}{\mathcal{X}}

\newcommand{\D}{\mathcal{D}}

\newcommand{\cH}{\mathcal{H}}

\newcommand{\M}{\mathcal{M}}

\newcommand{\s}{\subset}
\newcommand{\B}{\mathcal{B}}

\begin{document}

\title{Unitarizable representations and fixed points of groups of biholomorphic
transformations of operator balls}

\author{M.\,I.~Ostrovskii\\
Department of Mathematics and Computer Science\\
St. John's University\\
8000 Utopia Parkway\\
Queens, NY 11439\\
USA\\
e-mail: {\tt ostrovsm@stjohns.edu} \and\\
V.\,S.~Shulman\\
Department of Mathematics\\
Vologda State Technical University\\
15 Lenina street\\
Vologda 160000\\
RUSSIA\\
e-mail: {\tt shulman\_v@yahoo.com}
\and\\ L.~Turowska\\
Department of Mathematical Sciences\\
Chalmers University of Technology and University of Gothenburg\\
SE-41296, Gothenburg\\SWEDEN\\
e-mail: {\tt turowska@chalmers.se}}

\maketitle


\noindent{\bf Abstract.} We show that the open unit ball of the
space of operators from a finite dimensional Hilbert space into a
separable Hilbert space (we call it ``operator ball'') has a
restricted form of normal structure if we endow it with a
hyperbolic metric (which is an analogue of the standard hyperbolic
metric on the unit disc in the complex plane). We use this result
to get a fixed point theorem for groups of biholomorphic
automorphisms of the operator ball. The fixed point theorem is
used to show that a bounded representation in a separable Hilbert
space which has an invariant indefinite quadratic form with
finitely many negative squares is unitarizable (equivalent to a
unitary representation). We apply this result to find dual pairs
of invariant subspaces in Pontryagin spaces. In the appendix we
present results of Itai Shafrir about hyperbolic metrics on the
operator ball.
\bigskip

\bigskip


\section{Introduction}\label{S:Introduction}

Let $K,H$ be Hilbert spaces; by $L(K,H)$ we denote the Banach
space of all linear bounded operators from $K$ to $H$. We will
denote the open unit ball of $L(K,H)$ by $\B$ and call it {\it
operator ball}. We say that a subset $M$ of $\B$ is {\it separated
from the boundary} if it is contained in a ball $r\B$, for some
$r\in [0, 1)$.

A group $G$ of transformations of $\B$ will be called {\it elliptic} if all its orbits are separated from the
boundary (this terminology goes back to \cite{Helt}).

We call  $G$ {\it equicontinuous} if, for each $\varepsilon>0$ there is $\delta>0$ such that if $A,B\in \B$ and
$\|A-B\|< \delta$ then $\|g(A)-g(B)\|< \varepsilon$ for all $g\in G$. This condition can be also called {\it
global equicontinuity} because it is possible also to consider equicontinuity in a point.

Since $\B$ is a bounded open set of a Banach space, one may
consider holomorphic maps from $\B$ to Banach spaces. We will deal
with invertible holomorphic maps from $\B$ onto $\B$; such maps
are called {\it biholomorphic automorphisms} of $\B$. Our aim is
to prove that if one of the spaces $K, H$ is finite-dimensional
and the other is separable, then any elliptic group of
biholomorphic automorphisms of $\B$ has a common fixed point. More
precisely we will prove the following result.

\begin{theorem}\label{main}
Let $\dim K < \infty$ and $H$ be separable. For a group $G$ of
biholomorphic automorphisms of $\B$, the following statements are
equivalent:

{\rm (i)} $G$ is elliptic on $\B$;

{\rm (ii)} at least one orbit of $G$ is separated from the boundary;

{\rm (iii)} $G$ is equicontinuous;

{\rm (iv)} $G$ has a common fixed point in $\B$.

\end{theorem}

\begin{remark} The assumption $\dim K < \infty$ is essential, some
of the results of this paper are known to fail without it, see,
for example, the last paragraph of Section \ref{S:orthog}. As for
separability of $H$, it is just a technical convenience, our
approach works for non-separable $H$ also, with a bit more
complicated proofs.
\end{remark}

The result will be applied to the orthogonalization (or similarity) problem for bounded group representations on
Hilbert spaces. This problem can be formulated as follows. Let $\pi$ be a representation of a group $G$ on a
Hilbert space $\cH$. Under which conditions there is an invertible operator $V$ such that the representation
$\sigma$ of $G$, defined by the formula $\sigma(g) = V\pi(g)V^{-1}$, is unitary?

Clearly a necessary condition is the boundedness of $\pi$: $\sup_{g\in G}\|\pi(g)\|< \infty$. In general it is
not sufficient. Some sufficient conditions (on $G$ or $\pi$) are known, see the book \cite{Pis01}.
  We will show that a bounded representation
$\pi$ of a group $G$ on a Hilbert space $\cH$ is similar to a
unitary representation if it preserves a quadratic form $\eta$
with finite number of negative squares. The last condition means
that  $\eta(x) = \|Px\|^2 - \|Qx\|^2$ and $P,Q$ are orthogonal
projections in $\cH$ with $P+Q = 1$ and $\dim(Q\cH) < \infty$.
\medskip

As a consequence we obtain that each bounded group of $J$-unitary
operators on a Pontryagin space $\Pi_k$ has an invariant dual pair
of subspaces. In other words the space can be decomposed into
$J$-orthogonal direct sum  $H_+ + H_-$ of positive and negative
subspaces which are invariant for all operators in the group.
\medskip

The  proof of Theorem \ref{main} is based on the analysis of the
structure of the operator ball as a metric space with respect to
the Carath\'eodory distance (see Chapters 4 and 5 of
\cite{Vesent}). It was proved by Shafrir \cite{Sha} that $\B$ is a
hyperbolic space with respect to this distance. Since \cite{Sha}
is not easily accessible, we present a proof of this result in an
appendix to our paper, with the kind permission of the author. We
will show that $\B$ has a restricted form of a {\it normal
structure} if $\dim(K) < \infty$.

In the case where $K$ is one-dimensional Theorem \ref{main} was
obtained in \cite{Shul-80}; a transparent proof can be found in
\cite[Section 23]{Kiss-Shul}.

\section{Hyperbolic spaces}

In our definition of hyperbolic spaces we follow fixed point
theory literature (see e.g. \cite{RS90}, \cite{RZ01}). In
geometric literature (see e.g. \cite{BH99}) hyperbolic spaces are
defined differently.
\medskip

By a {\it line} in a metric space $(\X,\rho)$ we mean a subset of $\X$ which is isometric to the real line
$\mathbb{R}$ with its usual metric (in the literature lines are also called {\it metric lines} or {\it geodesic
lines}).
\medskip

Let $(\X,\rho)$ be a metric space with a distinguished set $\M$ of lines. We say that $\X$ is a {\it hyperbolic
space} if the following conditions are satisfied:
\medskip

\noindent{\bf (1)} (Uniqueness of a distinguished line through a given pair of points) For each $x,y\in \X$,
there is exactly one line $\ell\in\M$ containing both $x$ and $y$.
\medskip

\noindent{\bf (2)} (Convexity of the metric) To state the
condition (see \eqref{hyperb1}) we need to introduce some more
definitions and notation. The {\it segment} $[x,y]$ is defined as
the part of the line $\ell\in\M$ containing both $x$ and $y$,
consisting of all $z\in\ell$ satisfying
\begin{equation}\label{otrezok}
\rho(x,y) = \rho(x,z)+\rho(z,y). \end{equation} We write
\begin{equation}\label{otr1}
z=(1-t)x\oplus ty
\end{equation}
if $z\in[x,y]$, $\rho(z,x) = t\rho(x,y)$, and $\rho(z,y) =
(1-t)\rho(x,y)$ (where $t\in[0,1]$).

The convexity condition is:
\begin{equation}\label{hyperb1} \rho\left(\frac{1}{2}x\oplus\frac{1}{2}y,\frac{1}{2}x\oplus\frac{1}{2}z\right)\le
\frac{1}{2}\rho(y,z).
\end{equation}

Hyperbolic spaces satisfy also the following stronger form of the
condition (\ref{hyperb1}):
\begin{equation}\label{hyperb2}
\rho((1-t)x\oplus ty,(1-t)w\oplus tz)\le (1-t)\rho(x,w)+t\rho(y,z).
\end{equation}

(To get \eqref{hyperb2} from \eqref{hyperb1} we observe that, if
for some value of $t$ we have the inequalities $\rho((1-t)x\oplus
ty,(1-t)x\oplus tz)\le t\rho(y,z)$ and $\rho((1-t)x\oplus
tz,(1-t)w\oplus tz)\le (1-t)\rho(x,w)$, then, by the triangle
inequality, we have \eqref{hyperb2} for that value of $t$. Using
this observation repeatedly we prove the inequalities from this
paragraph for $t$ of the form $\frac{k}{2^n}$ $(k\in\mathbb{N},
1\le k\le 2^n)$. Then we use continuity.)
\medskip

A subset $C\s \X$ is called {\it convex} if $x,y\in C$ implies
$[x,y]\s C$. Sometimes we say {\it $\rho$-convex} instead of
convex, to avoid confusion with other natural notions of convexity
for the same set. We use the notation $E_{a,r}$ for $\{x\in
\X:~\rho(a,x)\le r\}$ and call such sets {\it closed balls}. The
condition \eqref{hyperb2} implies that in a hyperbolic space all
closed balls are convex.

\section{Normal structure}

Let $M$ be a subset in a metric space $(\X,\rho)$. The {\it diameter} of $M$ is defined by
\begin{equation}\label{diam} \diam M = \sup\{\rho(x,y): x,y\in M\}. \end{equation} A point $a\in M$ is called
{\it diametral} if
$$\sup\{\rho(a,x):x\in M\} = \diam M.$$

A hyperbolic space $\X$ is said to have {\it normal structure} if
every convex bounded subset of $\X$ with more than one element has
a non-diametral point.
\medskip

This notion goes back to Brodskii and Milman \cite{BM48} who
proved that uniformly convex Banach spaces (which are hyperbolic
spaces) have normal structure. Takahashi \cite{Tak} introduced and
studied normal structure in somewhat more general context. See
\cite[Chapter 3]{BL00} for a nice account on those aspects of
fixed point theory which are related to the geometry of Banach
spaces.

\begin{lemma}\label{perif}
Let $M$ be a separable bounded convex subset of a hyperbolic space $\X$ and $\alpha$ be the diameter of $M$. If
all points of $M$ are diametral then $M$ contains a sequence $\{a_n\}$ with the property: $\lim_{n\to \infty}
\rho(a_n,x) = \alpha$ for each $x\in M$.
\end{lemma}

\begin{proof}
Let $\{c_n\}$ be a dense sequence in $M$. We define a sequence
$\{b_n\}$ of ``centers of mass'' by the following rule: $b_1 =
c_1$, $b_{n+1} = \frac{n}{n+1}b_n\oplus \frac{1}{n+1}c_{n+1}$. By
convexity of $\rho$ we have
\begin{equation}\label{mean}
\rho(x,b_n) \le \frac{1}{n}\sum_{k=1}^n \rho(x,c_k)
\end{equation}
for all $n\in \mathbb{N}$. Indeed for $n=1$ this is obvious. If it
is true for some $n$, then $\rho(x,b_{n+1})\le
\frac{1}{n+1}\rho(x,c_{n+1}) + \frac{n}{n+1}\rho(x,b_n) \le
\frac{1}{n+1}\rho(x,c_{n+1}) +
\frac{n}{n+1}\frac{1}{n}\sum_{k=1}^n \rho(x,c_k) =
\frac{1}{n+1}\sum_{k=1}^{n+1} \rho(x,c_k)$.

By convexity of $M$ we have $b_n\in M$ for each $n\in\mathbb{N}$.
Our assumption implies that $b_n$ is diametral, hence there is a
point $a_n\in M$ with $\rho(b_n,a_n) \ge (1-
\frac{1}{n^2})\alpha$. It follows that $(1- \frac{1}{n^2})\alpha
\le \frac{1}{n}\sum_{k=1}^n \rho(a_n,c_k)$. If $\rho(a_n,c_j) <
(1-\frac{1}{n})\alpha$, for some $j\le n$, then
$\frac{1}{n}\sum_{k=1}^n \rho(a_n,c_k)<
\frac1n(1-\frac{1}{n})\alpha +\frac{n-1}n\alpha=(1-
\frac{1}{n^2})\alpha$, a contradiction. Hence $\rho(a_n,c_j) \ge
(1-\frac{1}{n})\alpha$ for $j\le n$. This shows that
$\lim_{n\to\infty}\rho(a_n,c_j)=\alpha$ for each fixed $j$. Since
the sequence $\{c_j\}$ is dense in $M$, the lemma is proved.
\end{proof}

\section{The invariant distance in the operator
ball}\label{S:distance}

Recall that $K,H$ denote Hilbert spaces and $\B$ is the open unit ball of $L(K,H)$. For $A,X\in \B$ set
\begin{equation}\label{mobius}
M_A(X) = (1-AA^*)^{-1/2}(A+X)(1+A^*X)^{-1}(1-A^*A)^{1/2}.
\end{equation}
Clearly all $M_A$ are holomorphic on $\B$. They are called {\it
the M\"obius transformations}.   It can be proved  that $M_A^{-1}
= M_{-A}$ (see \cite{harris}, Theorem 2). Hence each M\"obius
transformation is a biholomorphic automorphism of $\B$. Since
$M_A(0) = A$  the group of all biholomorphic automorphisms is
transitive on $\B$.
\medskip

We set
\begin{equation}\label{dist0}
\rho(A,B) = \tanh^{-1}( ||M_{-A}(B)||).
\end{equation}
It is easy to see   that $\rho$ coincides with the Carath\'eodory
distance $c_{\B}$ in $\B$. Indeed, by \cite[Theorem
4.1.8]{Vesent}, $c_{\B}(0,B) = \tanh^{-1}(\|B\|)$ (this holds for
the unit ball of every Banach space). Since $c_{\B}$ is invariant
and $M_A$ sends $A$ to $0$ we get:
\begin{equation}\label{dist}
c_{\B}(A,B) = \tanh^{-1} ||M_{-A}(B)|| = \rho(A,B).
\end{equation}
 Hence $\rho$ is invariant with respect to
biholomorphic automorphisms. I.~Shafrir \cite{Sha} proved that the
space $(\B,\rho)$ is hyperbolic. We present a proof of this result
in the appendix.
\medskip

A set in $\B$ is called {\it bounded} if it is contained in some $\rho$-ball, or equivalently in a multiple
$r\B$ of the operator ball with $r<1$. So a set is bounded if and only if it is separated from the boundary of
$\B$ in the sense of Section \ref{S:Introduction}.

The following lemma is a special case of a more general result
proved in \cite[Theorem IV.2.2]{Vesent}.
\begin{lemma}\label{twomet}
On any bounded set the hyperbolic metrics is equivalent to the operator norm.
\end{lemma}

\section{\WOT -topology}

As before, let $\B$ be the unit ball of the space of operators
from $K$ to $H$. We suppose that $K$ is finite-dimensional, $\dim
K = n$, and that $H$ is separable. We consider biholomorphic maps
on $\B$. By \WOT~ we denote the weak operator topology (see
\cite[p.~476]{DS58}). Because of the separability, the restriction
of this topology to $\B$ is metrizable, so in our arguments we may
consider only sequences, not nets.
\begin{lemma}\label{WOT}
If $K$ is finite-dimensional and $H$ is separable then all
biholomorphic maps of $\B$ are $\WOT $-con\-ti\-nuous.
\end{lemma}

\begin{proof} Let us firstly show that all M\"obius transforms $M_B$ are
$\WOT $-continuous (this was noticed and used already in the paper
of Krein \cite{Krein}). Indeed let $B\in \B$ be fixed, then the
map $\varphi: X\mapsto 1+B^*X$ from $(\B,\WOT )$ to
$(L(K,K),\WOT)$. Moreover, since $K$ is finite-dimensional,
$\varphi$ remains continuous if instead of \WOT~ we endow $L(K,K)$
with its norm topology. The map $T\to T^{-1}$ is norm continuous
on the group of invertible operators on $K$. Hence the map $\psi:
X\mapsto (1+B^*X)^{-1}$ is continuous from $(\B,\WOT)$ to $L(K,K)$
with its norm topology.

It follows that the map $\omega: X\to (X+B)(1+B^*X)^{-1}$ is
continuous from $(\B, \WOT)$ to $(\B, \WOT)$. Indeed, if $X_n\to
X$, then $\omega(X_n)-\omega(X) = (X_n +B)(\psi(X_n)-\psi(X)) +
(X_n - X)\psi(X)$, where $\psi$ was defined above. The first
summand tends to zero in norm while the second one tends to zero
in $\WOT$.
\medskip

By a result of Harris \cite{Har}, if a biholomorphic map of $\B$
preserves the point $0$, then it coincides with the restriction to
$\B$ of an isometric linear map $h: L(K,H)\to L(K,H)$. Since $K$
is finite-dimensional, the \WOT -topology on $L(K,H)$ coincides
with the weak topology (indeed $L(K,H)$ is linearly homeomorphic
to the direct sum of $n$ copies of $H$); since any bounded linear
map is weakly continuous, $h$ is \WOT-continuous. On the other
hand, if $\varphi$ is a biholomorphic map of $\B$ and $A =
\varphi(0)$ then $\psi = M_{-A}\circ \varphi$ is a biholomorphic
map preserving $0$. Hence $\psi$ is an isometric linear map and
$\varphi = M_{-A}^{-1}\circ \psi = M_A\circ \psi$ is a composition
of two \WOT-continuous maps. Thus $\varphi$ is \WOT-continuous.
\end{proof}

\begin{corollary}\label{loccompact}
If $\dim K < \infty$ and $H$ is separable, then each ball
$E_{A,r}$ is \WOT-compact.
\end{corollary}

\begin{proof}
Since there is a M\"obius transform that maps $E_{A,r}$ onto
$E_{0,r}$, and since all M\"obius transforms are \WOT-continuous,
it suffices to consider the case $A = 0$. But $E_{0,r}$ is a usual
closed operator ball; its \WOT-compactness follows from the
Banach-Alaoglu theorem.
\end{proof}

\section{Restricted normal structure of $\B$}

The purpose of this section is to show that in the case when $\dim
K <\infty$ and $H$ is separable, the (open) operator ball $\B$
with the metric \eqref{dist0} has a restricted form of normal
structure in the sense that \WOT-compact $\rho$-convex subsets in
it have non-diametral points. As we already mentioned $\B$ with
the metric \eqref{dist0} is a hyperbolic space (see Section
\ref{S:appendix}). Our assumptions on $K$ and $H$ imply that $\B$
is separable in the norm-topology and hence, by Lemma
\ref{twomet}, with respect to $\rho$.

\begin{theorem}\label{T:non-diam} Let $K$ be finite dimensional and $H$ be separable. Let $M$ be a weakly
compact, $\rho$-convex subset of $\B$ endowed with its hyperbolic
metric. If $M$ is not a singleton, then $M$ contains a
non-diametral point.
\end{theorem}

\begin{proof} Let $\alpha=\diam M>0$. Assume the contrary, that is, all points in $M$ are diametral.
By Lemma \ref{perif}, there is a sequence $\{A_n\}$ in $M$ such that $\lim_{n\to \infty} \rho(A_n,X) = \alpha$
for each $X\in M$.

 Since $M$ is weakly compact, the sequence $\{A_n\}_{n=1}^\infty$ contains a weakly convergent subsequence. Let
$W$ be its limit, we have $W\in M$ (since $M$ is weakly compact).
\medskip

Throughout this proof we will not change our notation after
passing to a subsequence.
\medskip

Since $W\in M$ we get
\begin{equation}\label{E:limrho}\lim_{n\to\infty}\rho(W,A_n)=\alpha.
\end{equation}

We will get a contradiction by proving
\begin{equation}\label{E:gamma}\sup_{n,m}\rho(A_n,A_m)>\alpha.\end{equation}

We may assume without loss of generality that $W=0$ (we can consider a M\"obius transformation which maps $W$ to
$0$, it is a $\rho$-isometry and weak homeomorphism).
\medskip

Let $\beta=\tanh\alpha$. Then \eqref{E:limrho} leads to
$\lim_{n\to\infty}||A_n||=\beta$ and it suffices to show that
$$\sup_{n,m}||M_{A_m}(-A_n)||>\beta.$$

Since $K$ is finite dimensional and $A_n\in L(K,H)$, we can select
a strongly convergent subsequence in the sequence $\{A_n^*A_n\}$.
Assume that  $A_n^*A_n\to P$, where $P\in L(K,K)$. It is clear
that $P\ge 0$ and $\|P\| = \beta^2$.
\medskip

Choose $\varepsilon> 0$ and fix a number $m$ with $\|A_m^*A_m -
P\|<\varepsilon $. For brevity, denote $A_m^*A_m$ by $Q$. We prove
that  $\lim_{n\to \infty}\|M_{A_m}(-A_n)\|
> \beta$ if $\varepsilon>0$ is small enough.
By the definition,
\begin{equation}\label{mobius2}
M_{A_m}(-A_n) = (1-A_mA_m^*)^{-1/2}(A_m-A_n)(1-A_m^*A_n)^{-1}(1-A_m^*A_m)^{1/2}.
\end{equation}
Since $A_m^*$ is of finite rank $A_m^*A_n \to 0$ in the norm topology. Hence $\lim_{n\to
\infty}\|M_{A_m}(-A_n)\| = \lim_{n\to \infty}\|T_n\|$ where
$$T_n = (1-A_mA_m^*)^{-1/2}(A_m-A_n)(1-A_m^*A_m)^{1/2} = A_m - (1-A_mA_m^*)^{-1/2}A_n(1-A_m^*A_m)^{1/2}.$$

It follows from the identity
$$(1-t)^{-1/2} - 1 = \frac{t}{(1-t)(1+(1-t)^{-1/2})}$$
that the operator $(1-A_mA_m^*)^{-1/2}$ is a finite rank
perturbation of the identity operator. Since $A_n\to 0$ in \WOT ,
we obtain that $\|T_n-S_n\|\to 0$, where $S_n = A_m -
A_n(1-A_m^*A_m)^{1/2}$.

Denote  $A_n(1-A_m^*A_m)^{1/2}$ by $B_n$. Since $B_n\to 0$ in \WOT
, the sequence $$(A_m - B_n)^*(A_m-B_n) - A_m^*A_m - B_n^*B_n = -
A_m^*B_n - B_n^*A_m$$  tends to zero in norm topology.
Furthermore,
$$B_n^*B_n = (1-Q)^{1/2}A_n^*A_n(1-Q)^{1/2} $$ tends in norm topology to $(1-Q)^{1/2}P(1-Q)^{1/2}$. Therefore
$$(A_m -
B_n)^*(A_m-B_n) \to Q + (1-Q)^{1/2}P(1-Q)^{1/2}.$$ Since $\|P-Q\|
<\varepsilon $, we have that $$\|Q + (1-Q)^{1/2}P(1-Q)^{1/2} - (Q
+(1-Q)Q)\| <\varepsilon.$$ The inequalities  $$\beta^2 -
\varepsilon\le \|Q\|\le \beta^2$$ imply $$\|Q + (1-Q)Q\| \ge
2\beta^2 - \beta^4 - 2\varepsilon,$$ whence
$$\lim_{n\to\infty}||S_n^*S_n||=\lim_{n\to\infty} \|(A_m -
B_n)^*(A_m-B_n)\|\ge 2\beta^2 - \beta^4 - 3\varepsilon > \beta^2,$$ if  $\varepsilon$ is sufficiently small.
\end{proof}

\section{Fixed points}

The main purpose of this section is to establish the existence of
a common fixed point for an elliptic group $G$ of biholomorphic
transformations of the operator ball $\B$. As is shown in Section
\ref{S:appendix} a biholomorphic transformation of $\B$ is a
bijective isometric transformation of the metric space $(\B,\rho)$
which maps the set $\M$ onto itself (and hence segments onto
segments).

\begin{lemma}\label{L:elliptic} If $G$ is an elliptic group of biholomorphic transformations of $\B$, then there is
a non-empty \WOT-compact $\rho$-convex $G$-invariant subset of
$\B$.
\end{lemma}

\begin{proof} Let $A\in \B$ be such that the orbit $G(A): = \{g(A):g\in G\}$ is bounded.
Therefore $G(A)$ is contained in some closed ball $E_{a,r}$. Let
$M$ be the intersection of all closed balls containing $G(A)$. It
is clear that this intersection is non-empty (it contains $G(A)$),
\WOT-compact and $\rho$-convex (as an intersection of \WOT-compact
$\rho$-convex sets). It remains to check that it is $G$-invariant.
To see this it suffices to observe that each element $g\in G$ maps
the set of balls containing $G(A)$ bijectively onto itself.
\end{proof}

\begin{lemma}\label{single} Let $G$ be an elliptic group of biholomorphic transformations of
$\B$. Let $M$ be a minimal \WOT-compact $\rho$-convex
$G$-invariant subset in $(\B,\rho)$. Then $M$ is a singleton.
\end{lemma}

\begin{proof} We use the approach suggested in \cite{BM48}. Assume the
contrary, let $\diam M=\alpha>0$. By Theorem \ref{T:non-diam} $M$
contains a non-diametral point $N$, so that
$M\subset\{A:~\rho(A,N)\le\delta\}$ for some $\delta<\alpha$.
Consider the set
$$O=\bigcap_{B\in M}E_{B,\delta}.$$
The set $O$ is non-empty because $N\in O$. The set $O$ is weakly compact and $\rho$-convex since each of the
balls $E_{B,\delta}$ is weakly compact  and $\rho$-convex. The set $O$ is a proper subset of $M$ since $M$ has
diameter $\alpha>\delta$.
\medskip

Since $G$ is a group of isometric transformations and $M$ is
invariant under each element of $G$, the action of $G$ on $M$ is
by isometric bijections. Therefore $O$ is $G$-invariant. We get a
contradiction with the minimality of $M$.
\end{proof}

\begin{proof}[{\bf Proof of Theorem \ref{main}}] The implication (i) $\Rightarrow$ (ii) is obvious. On the other
hand if $G(X_0)$ is separated from the boundary, for some $X_0\in \B$ then $\sup_{g\in G}\rho(0,g(X_0))< \infty$
whence, for each $X\in \B$, $\sup_{g\in G}\rho(0,g(X)) \le \sup_{g\in G}(\rho(0,g(X_0)) + \rho(g(X_0),g(X))) =
\sup_{g\in G}(\rho(0,g(X_0)) + \rho(X_0,X)) < \infty$. This means that the orbit $G(X)$ is separated from the
boundary. We proved that (i) $\Leftrightarrow$ (ii).
\medskip

The implication (i) $\Rightarrow$ (iv) can be derived from Lemmas
\ref{L:elliptic} and \ref{single} as follows. It is clear that
families of \WOT-compact $\rho$-convex $G$-invariant sets with the
finite intersection property have non-empty intersections which
are also \WOT-compact $\rho$-convex and $G$-invariant. Therefore,
by the Zorn Lemma, there is a minimal non-empty \WOT-compact
$\rho$-convex $G$-invariant set $M_0$. By Lemma \ref{single},
$M_0$ is a singleton and (iv) is proved.
\medskip

If (iv) is true and $A$ is a fixed point of $G$ then, $G_1 =
M_{-A}GM_{A}$ is a group of biholomorphic maps of $\B$ preserving
$0$. Hence it consists of restrictions to $\B$ of isometric linear
maps (see the beginning of Section \ref{S:distance} in this
connection). Thus $G_1$ is equicontinuous.

Note that each M\"obius transform is a Lipschitz map:
$\|M_A(X)-M_A(Y)\|\le C\|X-Y\|$ for each $X,Y\in \B$, where the
constant $C>0$ depends on $A$. Indeed setting $F(X) =
(A+X)(1+A^*X)^{-1}$ and $D = (1- \|A\|)^{-1}$ we have
\[\begin{split}\|F(X) -F(Y)\|&= \|(A+X)((1+A^*X)^{-1}- (1+A^*Y)^{-1}) + (X-Y)(1+A^*Y)^{-1}\|\\
& = \|(A+X)(1+A^*X)^{-1}A^*(Y-X)(1+A^*Y)^{-1}+(X-Y)(1+A^*Y)^{-1}\|\\
&\le 2D^2\|X-Y\| + D\|X-Y\|\le 3D^2\|X-Y\|.\end{split}\] Hence

$$\|M_A(X) -M_A(Y)\| = \|(1-AA^*)^{-1/2}(F(X) - F(Y))(1-A^*A)^{1/2}\|$$ $$ \le D^{\frac12}\|F(X)-F(Y)\|\le 3D^{\frac52}\|X-Y\|.$$

Since $G = M_AG_1M_{-A}$ and the maps $M_A$, $M_{-A}$ are Lipschitz, $G$ is also equicontinuous. We proved that
(iv) $\Rightarrow$ (iii).

Let now (iii) hold, we have to prove (ii). We will show that the orbit of $0$ is separated from the boundary.
Assuming the contrary we get that for any $\delta>0$ there is $g\in G$ with $\|g(0)\|>1-\delta$. Let $A = g(0)$;
we may assume that $\delta< 1/2$ so $\|A\|> 1/2$.

By the already mentioned result of \cite{Har}, $g = M_A\circ h$
where $h$ is a linear isometry. Let $P$ be the spectral projection
of $T = A^*A$ corresponding to the eigenvalue $\|T\|= \|A\|^2$
(recall that $T$ is an operator in a finite dimensional space).
Then
$$\|(1-T)P\| = 1-\|T\| \le 2(1- \|A\|) < 2\delta.$$ Set $X_1 = 0$,
$X_2 = h^{-1}(\frac{1}{2}AP)$. Then $\|X_2-X_1\| =
\frac{1}{2}\|AP\| = \|A\|/2 > 1/4$.

On the other hand
\[\begin{split}
\|g(X_2)-g(X_1)\|& = \left\|M_A\left(\frac{1}{2}AP\right) -
M_A(0)\right\|\\
&=
\left\|(1-AA^*)^{-1/2}\left(\frac{1}{2}AP+A\right)\left(1+\frac{1}{2}A^*AP\right)^{-1}(1-A^*A)^{1/2}-A\right\|\\
&=
\left\|A(1-T)^{-1/2}\left(\frac{1}{2}P+1\right)\left(1+\frac{1}{2}TP\right)^{-1}(1-T)^{1/2}
- A\right\|\\
&=
\left\|A\left(\frac{1}{2}P+1\right)\left(1+\frac{1}{2}TP\right)^{-1}
- A\right\| =
\left\|\frac{1}{2}A(1-T)P\left(1+\frac{1}{2}TP\right)^{-1}\right\|\\
& \le \frac{1}{2}\|A\|\|(1-T)P\| < \frac{1}{2}2\delta = \delta.\end{split}\]

This contradicts to the assumption of equicontinuity. Indeed for each $\delta$ we get points $Y_i = g(X_i)$ with
$\|Y_1-Y_2\|< \delta$ and $\|g^{-1}(Y_1) - g^{-1}(Y_2)\| > 1/4$. Thus (ii) holds.
\end{proof}

\section{Orthogonalization}\label{S:orthog}

\begin{theorem}\label{orth}
If a bounded representation $\pi$ of a group $G$ on a Hilbert space $\cH$ preserves a quadratic form $\eta$ with
finite number of negative squares then it is similar to a unitary representation.
\end{theorem}
\begin{proof}
By our assumptions, $\cH = H_1\oplus H_2$, $\dim(H_2) < \infty$,
and $\eta(x) = \|Px\|^2 - \|Qx\|^2$ where $P,Q$ are the
projections onto $H_1$ and $H_2$ respectively. We write $H_1 = H$
and $H_2 = K$, for brevity.\medskip

We will relate to each invertible operator $T$ on $\cH$ preserving the form $\eta$ a biholomorphic map $w_T$ of
$\B$ in such a way that
\begin{equation}\label{comp}
w_{T_1T_2} = w_{T_1}\circ w_{T_2}.
\end{equation}

Let us call a subspace $L$ of $\cH$ {\it positive} ({\it negative}) if $\eta(y) > 0$ (respectively $\eta(y) <
0$) for all non-zero $y\in L$. Since each negative subspace $L$ is finite-dimensional, there is $\varepsilon
> 0$ such that
$$\eta(y)
\le - \varepsilon\|y\|^2  \text{ for all non-zero }y\in L.$$ The
supremum of all such $\varepsilon$ is called the {\it degree of
negativeness} of $L$ and is denoted by $\epsilon(L)$.
\medskip

For each operator $A\in \B$, the set $$L(A) = \{Ax\oplus x: x\in K\}$$ is a negative subspace of $\cH$.
Furthermore the condition $$ \eta(y) \le - \varepsilon \|y\|^2, \text{ for all }y\in L(A)$$
 means that
$$ - \|x\|^2 + \|Ax\|^2 \le -\varepsilon (\|x\|^2 + \|Ax\|^2)$$
for all $x\in K$. That is $$ \varepsilon\le \frac{1-\|A\|^2}{1+\|A\|^2}.$$ It follows that the degree of
negativeness of $L(A)$ is related to $\|A\|$ by the equality
\begin{equation}\label{degree2}
\varepsilon(L(A)) = \frac{1-\|A\|^2}{1+\|A\|^2}.
\end{equation}

Since $\dim(L(A)) = \dim(K)$, $L(A)$ is a maximal negative
subspace in $\cH$. Indeed if some subspace $M$ of $\cH$ strictly
contains $L(A)$ then its dimension is greater than codimension of
$H$, whence $M\cap H \neq \{0\}$. But all non-zero vectors in $H$
are positive.

Conversely, each maximal negative subspace $Q$ of $\cH$ coincides
with $L(A)$, for some $A\in \B$. Indeed, since $Q\cap H = \{0\}$,
there is an operator $A: K\to H$ such that each vector of $Q$ is
of the form $Ax\oplus x$. Since $Q$ is negative, we have
$\eta(Ax\oplus x)=\|Ax\|^2-\|x\|^2<0$, and therefore $\|A\| < 1$,
so $A\in \B$. Thus $Q\subset L(A)$; and, by maximality, $Q =
L(A)$.

It is easy to see that the map $A\to L(A)$ from $\B$ to the set $\E$ of all maximal negative subspaces is
injective and therefore bijective.

Now we can define $w_T$. Indeed, if a subspace $L$ of $\cH$ is
maximal negative then its image $TL$ under $T$ is also maximal
negative (because $T$ is invertible and preserves $\eta$). Hence,
for each $A\in \B$, there is $B\in \B$ such that $L(B) = TL(A)$.
We let $w_T(A)=B$.

The equality (\ref{comp}) follows easily because $L(w_{T_1}(w_{T_2}(A))) = T_1L(w_{T_2}(A)) = T_1T_2L(A) =
L(w_{T_1T_2}(A))$  and the map $A\to L(A)$ is injective.

Our next goal is to check that $w_T$ is biholomorphic. Since $w_T^{-1} = w_{T^{-1}}$ it suffices to show that
$w_T$ is holomorphic.

Let $T = (T_{ij})_{i,j=1}^2 $ be the matrix of $T$ with respect to the decomposition $\cH = H_1\oplus H_2$. Then
$T(Ax\oplus x) = (T_{11}Ax + T_{12}x)\oplus(T_{21}Ax + T_{22}x)$. Since $T(Ax\oplus x) \in L(w_T(A))$, we
conclude that
$$w_T(A)(T_{21}Ax + T_{22}x) = T_{11}Ax + T_{12}x.$$
Thus
\begin{equation}\label{fraclin}
w_T(A) = (T_{11}A + T_{12})(T_{21}A + T_{22})^{-1}.
\end{equation}
This shows that $w_T$ is a holomorphic map on $\B$.

Suppose now that $\pi$ is a bounded representation of a group $G$ on $\cH$ preserving $\eta$. Then $W =
\{w_{\pi(g)}: g\in G\}$ is a group of biholomorphic maps of $\B$. Moreover since $\pi$ is bounded, the group $W$
is elliptic. To see this, note that for  each negative subspace $L$, one has
$$\eta(y)\le -\varepsilon(L)\|y\|^2 \text{ for all }y\in L.$$
If $T$ is an invertible operator preserving $\eta$ then $T^{-1}x\in L$, for each $x\in TL$, whence
$$\eta(x)= \eta(T^{-1}x) \le -\varepsilon(L)\|T^{-1}x\|^2 \le -\varepsilon(L)\|T\|^{-2}\|x\|^2.$$
Thus
$$\varepsilon(TL)\ge \varepsilon(L)\|T\|^{-2}.$$
For $L=L(A)$, $TL = L(w_T(A))$. This gives
$$\frac{1-\|w_T(A)\|^2}{1+\|w_T(A)\|^2} \ge  \|T\|^{-2} \frac{1-\|A\|^2}{1+\|A\|^2}$$
if one takes into account (\ref{degree2}). Thus,  if
$\|\pi(g)\|\le C$ for all $g\in G$, then
$$\frac{1-\|w_{\pi(g)}(A)\|^2}{1+\|w_{\pi(g)}(A)\|^2} \ge  C^{-2} \frac{1-\|A\|^2}{1+\|A\|^2}.$$
Therefore $$ 1-\|w_{\pi(g)}(A)\|^2 \ge C^{-2}
\frac{1-\|A\|^2}{1+\|A\|^2}$$ and
 $$\sup_{g\in G}\|w_{\pi(g)}(A)\| < 1$$
 for each $A\in \B$.

By Theorem \ref{main}, there is $D\in \B$ with $w_{\pi(g)}(D) = D$ for all $g\in G$. Hence $\pi(g)L(D) = L(D)$
for all $g\in G$.

Let $U$ be an operator on $\cH$ with the matrix $(U_{ij})$ where $U_{11} = (1_H-DD^*)^{-1/2}$, $U_{12} = -
D(1_K-D^*D)^{-1/2}$, $U_{21} = -D^*(1_H-DD^*)^{-1/2}$, $U_{22} = (1_K-D^*D)^{-1/2}$.  It can be checked that $U$
preserves  $\eta$ and maps $L(D)$ onto $K$. Then all operators $\tau(g) = U\pi(g)U^{-1}$ preserve $\eta$, and
the subspace $K$ is invariant for them. It follows that $H$ is also invariant for operators $\tau(g)$. Hence
these operators preserve the scalar product on $\cH$. Thus $g\mapsto \tau(g)$ is a unitary representation
similar to $\pi$.
\end{proof}

It should be noted that Theorem \ref{orth} does not extend to the
case when $\eta$ has infinite number of negative (and positive)
squares, that is, to the case that both $H_1$ and $H_2$ are
infinite-dimensional \cite{Shul-77}.

\section{$J$-unitary operators on Pontryagin spaces}

The Pontryagin space is a linear space $\cH$ supplied with an
indefinite scalar product $x,y\to [x,y]$ which has a finite number
of negative squares. More precisely this means that one can choose
a usual scalar product $x,y \to (x,y)$ with respect to which $\E$
is a Hilbert space and $[x,y] = (Jx,y)$, where $J$ is a selfadjoint
involutive operator on this Hilbert space with $\text{rank} (1-J)
< \infty$. An invertible operator $T$ on $\E$ is called $J$-{\it
unitary} if $[Tx,Ty] = [x,y]$ for all $x,y\in \E$.

It should be noted that the terminology does not seem to be
successful because the choice of the operator $J$ and the
corresponding scalar product is not unique while the set of
$J$-unitary operator is completely determined by the original
indefinite scalar product $[\cdot,\cdot]$. However, this
terminology is widely used (see, for example, \cite{Az-Iohv},
\cite{Kiss-Shul} and references therein). It is important that all
scalar products defining $[\cdot,\cdot]$ via $J$-operators are
equivalent, so one can speak, for example, about boundedness of a
set of operators, without indicating which scalar product is
chosen.

A subspace $X\subset \E$ is called {\it positive} ({\it negative})
if $[x,x]
> 0$ (respectively $[x,x] < 0$) for all $x\in X$. A {\it dual pair
of subspaces} in $\E$ is a pair $Y,Z$, where $Y$ is a positive
subspace, $Z$ is a negative subspace and $Y+Z = \E$. The study of
dual pairs invariant for a given set of $J$-unitary operators was
started by Sobolev and intensively developed by Pontryagin, Krein,
Phillips, Naimark and other prominent mathematicians.

The previous theorem on the orthogonalization of representations
implies the following result.

\begin{corollary}\label{Pontr}
A group of $J$-unitary operators on a Pontryagin space has an invariant dual pair if and only if it is bounded.
\end{corollary}
\begin{proof}
Choose a scalar product $(\cdot,\cdot)$ and the corresponding
operator $J$. Denote by $\cH$ the Hilbert space $(\E,(\cdot , \cdot))$.
Since $J$ is an Hermitian involutive operator, there are
orthogonal subspaces $H$, $K$ of $\cH$ such that $J = P_H - P_K$.
By our assumption on $J$, the subspace $K$ is finite-dimensional.

Let $G$ be a group of $J$-unitary operators. If it is bounded,
then the identity map can be regarded as a bounded representation
of $G$ on $\cH$. Moreover it preserves the form $\eta(x) = [x,x]$.
Since it has a finite number of negative squares, Theorem
\ref{orth} implies that there is an invertible operator $V$ such
that the representation $\tau(g) = T^{-1}gT$ is unitary. It
follows from \cite[Theorem 5.8]{Kiss-Shul} that $G$ has an
invariant dual pair of subspaces.
\medskip

For completeness we include the proof of this fact. Passing to
adjoints in the equality $T\tau(g) = gT$ and taking into account
that $g^* = Jg^{-1}J$, $\tau(g)^* = \tau(g^{-1})$ we obtain that
$\tau(g^{-1})T^* = T^*Jg^{-1}J$. Using this identity for $g$
instead of $g^{-1}$ and multiplying both sides by $JT$ we get:
$$\tau(g)T^*JT = T^*JgJJT = T^*JgT = T^*JT\tau(g).$$

Thus the invertible selfadjoint operator $R = T^*JT$ commutes with
the group $\tau(G)$ of unitary operators. It follows that its
spectral subspaces $H_1$ and $K_1$ corresponding to positive and
negative parts of spectrum are invariant for $\tau(G)$. Note that
$(Rx,x) > 0$ for $x\in T^{-1}H\backslash\{0\}$ and $(Rx,x) < 0$
for $x\in T^{-1}K\backslash\{0\}$. It follows that $\dim K_1 =
\dim K$. Now the subspaces $H_2 = TH_1$ and $K_2 = TK_1$ form an
invariant dual pair for $G$.

The converse implication is simple. If $G$ has an invariant dual
pair $H,K$ then the scalar product $(h_1+k_1,h_2+k_2) = [h_1,h_2]
- [k_1,k_2]$ is invariant for $G$. Thus $G$ is a group of unitary
operators on $\cH = (\E, (\cdot,\cdot))$, hence it is bounded.
\end{proof}

As a consequence we obtain the following result proved in
\cite{Shul-77}:
\begin{corollary}
A $J$-symmetric representation of a unital $C^*$-algebra on a
Pontryagin space is similar to a ${}^*$-representation.
\end{corollary}
For a proof it suffices to notice that restricting the
representation to the unitary group of the $C^*$-algebra we obtain
a bounded group of $J$-unitary operators.

\section{Appendix: Hyperbolicity of $\B$ (after  Itai
Shafrir)}\label{S:appendix}

For any bounded domains $D_1,D_2$ of complex Banach spaces we
denote by $Hol(D_1,D_2)$ the set of all holomorphic maps from
$D_1$ to $D_2$. If $D_1 = D_2 = D$ then $Hol(D_1,D_2)$ is a
semigroup with respect to the composition, and by $Aut(D)$ we
denote the set of all its invertible elements (biholomorphic
automorphisms of $D$). The group $Aut(\B)$ acts transitively on
$\B$. Indeed, for each $A\in \B$ the M\"obius transform $M_A$ is
biholomorphic and sends $0$ to $A$

As usually the Carath\'eodory metric on $\B$ is defined by the
equality:
$$c_{\B}(A,B) = \sup\{\omega(f(A),f(B)): f\in Hol(\B,\Delta)\}$$
where $\Delta$ is the unit disk and $\omega$ is the Poincar\'e
distance:
$$\omega(z_1,z_2) = \tanh^{-1}\left|\frac{z_1-z_2}{1-\overline{z_1}z_2}\right|.$$
As it was mentioned in Section 4, $c_{\B}$ coincides with the metric $\rho$ defined by the formula
(\ref{dist0}). Clearly $c_{\B}$ is invariant under biholomorphic maps of $\B$.

We shall prove that $\B$ is a hyperbolic space with respect to
this metric.

Furthermore the differential Carath\'eodory metrics on $\B$ is
defined by
\begin{equation}\label{metr}
\alpha(A,V) = \sup_{f\in Hol(\B,\Delta)}\frac{|\D f(A)V|}{1-|f(A)|^2}
\end{equation}
for all $A\in \B, V\in L(K,H)$, where $\D f(A)$ is the
differential of $f$ in $A$ (see \cite{Vesent}, where $\alpha$ is
denoted by $\gamma_{\B}$).

\bigskip

\begin{lemma}\label{differential}
For each $A\in\B$, $V\in L(K,H)$
\begin{equation}\label{diff}
\D M_B(A)V =
(1-BB^*)^{1/2}(1+AB^*)^{-1}V(1+B^*A)^{-1}(1-B^*B)^{1/2}.
\end{equation}
In particular,
$$\D M_B(0)V = (1-BB^*)^{1/2}V(1-B^*B)^{1/2}.$$
\end{lemma}

\begin{proof}
By definition, $M_B(X) =
(1-BB^*)^{-1/2}(B+X)(1+B^*X)^{-1}(1-B^*B)^{1/2}$. We have to
calculate the coefficient $c$ of $t$ in the Taylor decomposition
of the function $t\to M_B(A+tV)$. For this, note that if $P$ is an
invertible operator then $(P+tQ)^{-1} = P^{-1} -tP^{-1}QP^{-1} +
o(t)$. It follows immediately that \[\begin{split}c
&=(1-BB^*)^{-1/2}(V(1+B^*A)^{-1} -
(B+A)(1+B^*A)^{-1}B^*V(1+B^*A)^{-1})(1-B^*B)^{1/2}\\
&=(1-BB^*)^{-1/2}(1-(B+A)(1+B^*A)^{-1}B^*)V(1+B^*A)^{-1}(1-B^*B)^{1/2}\\
&=(1-BB^*)^{-1/2}(1-(B+A)B^*(1+
AB^*)^{-1})V(1+B^*A)^{-1}(1-B^*B)^{1/2}\\
&=(1-BB^*)^{-1/2}((1+AB^*-(B+A)B^*)(1+AB^*)^{-1})V(1+B^*A)^{-1}(1-B^*B)^{1/2}\\
&=(1-BB^*)^{1/2}(1+AB^*)^{-1}V(1+B^*A)^{-1}(1-B^*B)^{1/2}.\end{split}\]
\end{proof}

\begin{lemma}\label{alpha}
$\alpha(A,V) = \|(1-AA^*)^{-1/2}V(1-A^*A)^{-1/2}\|$ for all $A\in \B$ and $V\in L(K,H)$.
\end{lemma}
\begin{proof}
By \cite[Lemma V.1.5]{Vesent}
$$\alpha(0,V)= \|V\|.$$
Let now $A$ be arbitrary. Then by \cite[Proposition~V.1.2]{Vesent}
$$\alpha(M_A(0),\mathcal D M_A(0)X)=\alpha(0,X)=||X||.$$ On the
other hand, by Lemma~\ref{differential},
$$\alpha(M_A(0),\mathcal D M_A(0)X)=\alpha(A,
(1-AA^*)^{1/2}X(1-A^*A)^{1/2}).$$ Setting now
$V=(1-AA^*)^{1/2}X(1-A^*A)^{1/2}$, we obtain
$X=(1-AA^*)^{-1/2}V(1-A^*A)^{-1/2}$ and hence
$\alpha(A,V)=||(1-AA^*)^{-1/2}V(1-A^*A)^{-1/2})||$. \end{proof}


For any bounded operator $D$, set $D^{(1)} = D$, $D^{(3)} = DD^*D$,
$D^{(5)} = DD^*DD^*D,\ldots,$ $D^{(2k+1)} = (DD^*)^kD$.

Let
\begin{equation}\label{Th}
\text{Th }D = \sum_{n=0}^{\infty}a_{2n+1}D^{(2n+1)}
\end{equation}
where $a_j$ are the Taylor coefficients of $\tanh t$, i.e.,  $\tanh
t = \sum_{n=0}^{\infty}a_{2n+1}t^{2n+1}$.

It follows from the definition that $\text{Th }D = \tanh(D)$ if
$D$ is selfadjoint.

If $D = J|D|$ is the polar decomposition of $D$ (that is,
$|D|=(D^*D)^{1/2}$ and  $J$ is a partial isometry such that $(J^*J)|D| = |D| (J^*J) = |D|$), then
$$D^{(2n+1)} = J|D|^{2n+1},$$ and
hence
$$\text{Th }D = J\tanh |D|.$$
On the other hand, we can write $D = |D^*|J$, where $|D^*| =
J|D|J^*=(DD^*)^{1/2}$, therefore
$$\text{Th }D = (\tanh |D^*|)J.$$

\bigskip
For the  space $(\B,\rho)$ we define the set $\M$ of lines as
follows: for $A\in \B$, $D\in \partial \B$ (i.e., $||D||=1$) we
let
\begin{equation}\label{line }
\gamma_{A,D} = \{\gamma_{A,D}(t):=M_A(\text{Th}(tD)):t\in \mathbb{R}\}
\end{equation}
and set
$$\M=\{\gamma_{A,D}:A\in\B,D\in \partial \B\}.$$
\begin{proposition}\label{metricline}
$\gamma_{A,D}$ is a metric line.
\end{proposition}
\begin{proof}
It is enough to show that
$\rho(\gamma_{A,D}(s),\gamma_{A,D}(t))=|s-t|$. Since $\rho$ is
invariant with respect to $M_A$ we can assume that $A=0$. We have
$\rho(\gamma_{0,D}(s),\gamma_{0,D}(t))=\tanh^{-1}||M_B(\text{Th}(tD))||$,
where $B=-\text{Th}(sD)$. Using polar decomposition $D = J|D|$ we
have that $\text{Th}(tD) = J \tanh(t|D|)$, $\text{Th }(tD)^*\text{Th }(sD)=\tanh (t|D|)\tanh (s|D|)$, whence \marginpar{!}
\begin{eqnarray*}
M_B(\text{Th}(tD)) &=&
(1-\text{Th}(sD)\text{Th}(sD)^*)^{-1/2}(\text{Th}(tD)-\text{Th}(sD))\\
&&(1-\text{Th}(sD)^*\text{Th}(sD))^{-1}
(1-\text{Th}(sD)^*\text{Th}(sD))^{1/2}\\
&=& J(1-\tanh^2(s|D|))^{-1/2}J^*J(\tanh(t|D|) - \tanh(s|D|))\\ &&(1-\tanh(s|D|)\tanh(t|D|))^{-1}(1-\tanh^2(s|D||))^{1/2}\\
&=& J \tanh((t-s)|D|) = \text{Th}((t-s)D)
\end{eqnarray*}
giving the statement.
\end{proof}

We have to prove that $\gamma_{A,D}(t)$ is a {\it metric curve},
in the sense that the metric of its derivation equals $1$, i.e.,
$\alpha(\gamma(t),\gamma^{\prime}(t)) = 1$.

\begin{lemma} Let
 $\gamma(t) =
\text{\rm Th}(tD)$, $D\in \partial \B$. Then
\begin{equation}\label{dif eq }
\gamma^{\prime}(t) = D - \gamma(t)D^*\gamma(t).
\end{equation}
\end{lemma}


\begin{proof}
We have $\gamma(t)=J\tanh(t|D|)$ and
$$\gamma'(t)=J|D|\cosh(t|D|)^{-2}= D(\cosh(t|D|)^{-2}.$$ On the
other hand
$$D-\gamma(t)D^*\gamma(t)=D-J\tanh(t|D|)|D|J^*J\tanh(t|D|)=D-D\tanh^2(t|D|)=D(\cosh(t|D|))^{-2}$$ giving (\ref{dif eq }).
\end{proof}

\begin{lemma}\label{met} Let $\gamma(t)=\text{\rm Th}(tD)$, $D\in\partial\B$. Then
\begin{equation}\label{metrcurve}
(1-\gamma\gamma^*)^{-1/2}(D-\gamma
D^*\gamma)(1-\gamma^*\gamma)^{-1/2} = D.
\end{equation}
\end{lemma}
\begin{proof}
Setting $D = J|D|$, we have
$$\gamma^*\gamma = \text{Th}(tD)^*\text{Th}(tD) = \tanh(t|D|)J^*J \tanh(t|D|)
= \tanh^2(t|D|).$$ Furthermore
$$\gamma\gamma^* = \tanh(t|D^*|)JJ^*\tanh(t|D^*|) = \tanh^2(t|D^*|).$$
Since
$$D|D| = |D^*|D,$$
we have that
$$D f(|D|) = f(|D^*|)D$$
for any bounded Borel function $f$. Taking $f(x) = 1-\tanh^2(tx)$,
we get
$$D(1-\gamma^*\gamma) = (1-\gamma\gamma^*)D.$$

Next
$$D^*\gamma = \gamma^*D$$
because $D^*\text{Th}(tD) = D^*J \tanh(t|D|) = |D| \tanh(t|D|)$ is
selfadjoint.

Hence
\begin{eqnarray*}
(1-\gamma\gamma^*)^{-1/2}(D-\gamma
D^*\gamma)(1-\gamma^*\gamma)^{-1/2} =
(1-\gamma\gamma^*)^{-1/2}(1-\gamma\gamma^*)D(1-\gamma^*\gamma)^{-1/2} \\
=(1-\gamma^*\gamma)^{1/2}(1-\gamma^*\gamma)^{-1/2}D = D.
\end{eqnarray*}
\end{proof}

\begin{proposition}\label{C:metrcurve} Let $\gamma(t)=\gamma_{A,D}(t)$, $A\in\B$, $D\in\partial \B$. Then
$$\alpha(\gamma(t),\gamma^{\prime}(t)) = 1.$$
\end{proposition}
\begin{proof}
It suffices to prove this for $A
= 0$, since $\alpha(F(X),\D F(X)V)=\alpha(X,V)$ for any $F\in
Aut(\B)$, $X\in\B$, $V\in L(H,K)$ (see
\cite[Proposition~V.1.2]{Vesent}), and hence
\begin{eqnarray*}
\alpha(M_A(\text{Th}(tD)),
(M_A(\text{Th}(tD)))')&=&\alpha(M_A(\text{Th}(tD)),\D
M_A(\text{Th}(tD))(\text{Th}(tD))')\\
&=&\alpha(\text{Th}(tD),(\text{Th}(tD))').
\end{eqnarray*}
Assume therefore $\gamma(t)=\text{Th}(tD)$.
By Lemma~\ref{alpha} and \ref{met} we have
$$\alpha(\gamma(t),\gamma^{\prime}(t)) =
\|(1-\gamma\gamma^*)^{-1/2}(D-\gamma
D^*\gamma)(1-\gamma^*\gamma)^{-1/2}\| = \|D\| = 1.$$
\end{proof}

The next step is to prove that the family $\M$ of all lines is
invariant with respect to the biholomorphic maps of $\B$.
\begin{lemma}\label{invbihol}
Let $\eta(t) = M_A(\gamma(t))$ where $\gamma(t) = \text{\rm
Th}(tD)$. Then, for each biholomorphic map $h:\B\to \B$, the curve
$h(\eta(t))$ belongs to the family $\M$.
\end{lemma}
\begin{proof}
By  \cite[Theorems~3 and 4]{harris}, there is a linear isometry
$L$ of the space $L(K,H)$ to itself satisfying the condition
\begin{equation}\label{multipl}
L(AB^*A) = L(A)L(B)^*L(A)\text{  for all  } A,B\in L(K,H)
\end{equation} and
such that $$h = M_{h(0)}\circ L = L\circ M_{-h(0)}.$$ It follows
from (\ref{multipl})  (see  a remark after \cite[Corollary
~5]{harris}) that $$L\circ M_A = M_{L(A)}\circ L$$ for all $A\in\B$.

So it suffices to consider the cases $h=L$ and $h = M_B$.
 Let us firstly prove that $L(\eta(t))\in \M$. Indeed,
\[\begin{split}L(\eta(t)) &= L(M_A(\gamma(t)) = M_{L(A)}(L(\gamma(t))
\\&=M_{L(A)}(L(\text{Th}(tD))) = M_{L(A)}(\text{Th}(tL(D)))\in
\M.\end{split}\] Now we have to prove that $M_B(\eta(t))\in \M$.
Applying \cite[Theorems~3 and 4]{harris} to  $h(x)= M_B(M_A(x))$
we get a linear isometry $L$ satisfying (\ref{multipl}) and such that
$$M_B\circ M_A = M_C\circ L$$
where $C = h(0)=M_B(A)$.
 Thus $$M_B(\eta(t)) = M_B(M_A(\gamma(t)) = M_C(L(\gamma(t)))=M_C(\text{Th}(tL(D))) \in \M.$$
\end{proof}

\bigskip

Our next goal is to show that for each $A,B\in\B$ there is a unique
line in $\M$ which passes through $A$, $B$.
\begin{lemma}\label{viaA}
The set of all lines in $\M$ that go through $A$ is
$\{\gamma_{A,D}: D\in
\partial(\B)\}$.
\end{lemma}
\begin{proof}
It suffices to assume that $A = 0$.
Suppose that a line $\gamma(t)
= M_B({\rm Th}(tD))$ goes through $0$, i.e., $\gamma(s) = 0$ for
some $s\in{\mathbb R}$. Then clearly $B = - \text{Th}(sD)$. Using
the arguments from the proof of Proposition~\ref{metricline} we
obtain $\gamma(t)=
 \text{Th}((t-s)D).$
Thus $\gamma=\gamma_{0,D}$.
\end{proof}

\begin{corollary}\label{unique}
For each $A,B\in \B$, there is a unique  line in $\M$ that passes
through them.
\end{corollary}
\begin{proof}
We may assume that $A = 0$. Let $B=J|B|$ be the polar
decomposition of $B$ and let $C=\tanh^{-1}|B|/t_0$ for $t_0>0$ be
such that $||C||=1$.
 Then for $D=JC$ the line $\gamma_{0,D}$ passes through $0$ and $B$.

If there are two lines, $\gamma_{0,D_1}$ and $\gamma_{0,D_2}$, going through
$B$ then by the above lemma, $B = \text{Th}(tD_1)
=\text{Th}(sD_2)$ for some $t$, $s\in{\mathbb R}$. We may suppose that $t$, $s > 0$. Taking polar
decompositions of $D_1=J_1|D_1|$ and $D_2=J_2|D_2|$ we see that
$J_1=J_2$ and $\tanh(t|D_1|)=\tanh(s|D_2|)$,  which imply that
$t|D_1|=s|D_2|$. But this clearly shows that the lines coincide.
\end{proof}

\begin{lemma}\label{inequality}
\begin{equation}\label{ineq}
\|A\|\le \|(1-BB^*)^{-1/2}(A - BA^*B)(1-B^*B)^{-1/2}\|.
\end{equation}
 for each $A$, $B\in\B$.
\end{lemma}
\begin{proof}
Consider the polar decomposition $B=J|B|$. Then
$|B^*|=(BB^*)^{1/2}=J|B|J^*$. Let $P=\tanh^{-1}(|B^*|)$, and
$Q=\tanh^{-1}(|B|)$. Then
$$(1-BB^*)^{-1/2}(A - BA^*B)(1-B^*B)^{-1/2}=(\cosh P)A(\cosh Q)-
(\sinh P)JAJ^*(\sinh Q).$$


For any $\varepsilon>0$, there are unit vectors $x,y$ such that
$$((\cosh P)A(\cosh Q)x,y)\ge \|(\cosh P)y\|\|A\|\|(\cosh Q)x\| -\varepsilon.$$
Since $\|(\cosh P)y\|^2-\|(\sinh P)y\|^2 = \|y\|^2$, and
 $\|(\cosh Q)x\|^2-\|(\sinh Q)x\|^2 = \|x\|^2$
one can find numbers $a,b$ such that

$\|(\sinh P)y\|= \sinh b$, $\|(\cosh P)y\|= \cosh b$, $\|(\sinh
Q)x\|= \sinh a$, $\|(\cosh Q)x\|= \cosh a$.

Hence
\begin{eqnarray*}
&&\|(\cosh P)A(\cosh Q)-
(\sinh P)JAJ^*(\sinh Q)\| \\
&&\ge (((\cosh P)A(\cosh Q)-
(\sinh P)JAJ^*(\sinh Q))x,y) \\
&&\ge (\cosh b)(\cosh a)\|A\| - \varepsilon - (\sinh b)(\sinh a)\|A\|\\
&&\ge \cosh (b-a)\|A\| - \varepsilon\ge \|A\| - \varepsilon,
\end{eqnarray*}
giving the statement.
\end{proof}

\begin{lemma}\label{lem-hyperb}
Let us consider two lines: $\gamma(t) = M_A(\text{\rm Th}(tC))$,
$\eta(t) = M_A(\text{\rm Th}(tD))$. Then
\begin{equation}\label{hyp}
2\rho(\gamma(s),\eta(s))\le \rho(\gamma(2s),\eta(2s))
\end{equation}
for each $s>0$.
\end{lemma}
\begin{proof}
Since $\rho$ is invariant with respect to the transformations $M_A$ we
may assume
 $A = 0$.


Let $C(t)$ be a  curve $\gamma_{B,E}(t)$ which  joins $\gamma(2s)$
with $\eta(2s)$, we assume that  $C(0) = \gamma(2s)$, $C(t_0) =
\eta(2s)$ for some $t_0>0$ (such curve  exists by
Corollary~\ref{unique}). Define now a new curve $C_1$ by
$$C_1 =
\text{Th}\left(\frac{1}{2}\text{Th}^{-1}C\right).$$ Then  $C_1(0)=\gamma(s)$,
$C_1(t_0)=\eta(s)$ and
\begin{equation}\label{tanh}
C(t)=2C_1(t)(1+C_1(t)^*C_1(t))^{-1}.
\end{equation}
 As usually we denote by $L(C_1)$  the length of the curve
$C_1$: $L(C_1) = \int_0^{t_0}\alpha(C_1(t),C_1^{\prime}(t))dt$.

 If we could
show that
\begin{equation}\label{length}
L(C)\ge 2L(C_1)
\end{equation}
for  all curves $C$, $C_1$ satisfying (\ref{tanh})
then we  would obtain that
$$\rho(\gamma(2s),\eta(2s)) = L(C)\ge 2L(C_1) \ge 2\rho(\gamma(s),\eta(s))$$
(the first equality follows from Proposition~\ref{metricline} and
Proposition~\ref{C:metrcurve},  the last inequality holds because
the length of any curve is not smaller then the distance between
its ends).

 Thus our goal is the inequality (\ref{length}).
It suffices to show that
\begin{equation}\label{alphaineq}
2\alpha(C_1(t),C_1^{\prime}(t))\le \alpha(C(t),C^{\prime}(t)).
\end{equation}
Since $$2C_1 = C(1+C_1^*C_1),$$ we have
$$C^{\prime}(1+C_1^*C_1) + C(C_1^{\prime *}C_1+C_1^*C_1^{\prime}) = 2C_1^{\prime},$$
whence
\begin{equation}\label{Cprime}
C^{\prime} = ((2-CC_1^*)C_1^{\prime} - CC_1^{\prime
*}C_1)(1+C_1^*C_1)^{-1}.
\end{equation}
Since
$$2-CC_1^* = 2-2C_1(1+C_1^*C_1)^{-1}C_1^* = 2(1-C_1C_1^*(1+C_1C_1^*)^{-1}) = 2(1+C_1C_1^*)^{-1},$$
substituting this into (\ref{Cprime}) we obtain
$$C^{\prime} = 2((1+C_1C_1^*)^{-1}C_1^{\prime} - C_1(1+C_1^*C_1)^{-1}C_1^{\prime *}C_1)(1+C_1^*C_1)^{-1}$$
$$= 2(1+C_1C_1^*)^{-1}(C_1^{\prime} - C_1C_1^{\prime *}C_1)(1+C_1^*C_1)^{-1}.$$
Now it follows from Lemma \ref{alpha} that the inequality
(\ref{length}) is equivalent to the following
\begin{equation}\label{reform}\begin{split}
&\|(1-C_1C_1^*)^{-1/2}C_1^{\prime}(1-C_1^*C_1)^{-1/2}\|\\
&\le\|(1-CC^*)^{-1/2}(1+C_1C_1^*)^{-1}(C_1^{\prime}-C_1C_1^{\prime
*}C_1)(1+C_1^*C_1)^{-1}(1-C^*C)^{-1/2}\|.
\end{split}
\end{equation}
But
\[\begin{split}& 1-CC^* = 1-4C_1(1+C_1^*C_1)^{-2}C_1^* = 1- 4
C_1C_1^*(1+C_1C_1^*)^{-2}\\
&=((1+C_1C_1^*)^2-4C_1C_1^*)(1+C_1C_1^*)^{-2} =
(1-C_1C_1^*)^2(1+C_1C_1^*)^{-2}.\end{split}\] Similarly
$$(1-C^*C)^{-1/2} = (1+C_1^*C_1)(1-C_1^*C_1)^{-1}.$$
It follows now that (\ref{reform}) is equivalent to the inequality
\begin{equation}\label{newineq}
\|(1-C_1C_1^*)^{-1/2}C_1^{\prime}(1-C_1^*C_1)^{-1/2}\| \le
\|(1-C_1C_1^*)^{-1}(C_1^{\prime }-C_1C_1^{\prime
*}C_1)(1-C_1^*C_1)^{-1}\|.
\end{equation}
But (\ref{newineq}) follows from Lemma \ref{inequality} by
substituting $B = C_1$ and $A =
(1-C_1C_1^*)^{-1/2}C_1^{\prime}(1-C_1^*C_1)^{-1/2}$ into
inequality (\ref{ineq}).
\end{proof}
\bigskip

The above results establish

\begin{theorem}\label{hypex}
$\B$ is a hyperbolic space.
\end{theorem}

\noindent{\bf Acknowledgements.} We wish to express our gratitude
to Professor Itai Shafrir for informing us about results of his
dissertation \cite{Sha} and to Ekaterina Shulman for providing us
with a copy of \cite{Sha} and for helping us with its translation.
The second author also would like to thank Alexei Loginov and
Natal'ya Yaskevich for helpful discussions on the subject of this
paper many years ago.

\end{document}